\documentclass[12pt]{article}
\usepackage{graphicx}
\usepackage{amsmath}
\usepackage{amssymb}
\usepackage{theorem}
\usepackage{enumerate} 
\usepackage{color}

  \usepackage{hyperref}

\numberwithin{equation}{section}

\newtheorem{thm}{Theorem}[section]
\newtheorem{lemma}[thm]{Lemma}

\newtheorem{prop}[thm]{Proposition}
\newtheorem{cor}[thm]{Corollary}
{\theorembodyfont{\rmfamily}
\newtheorem{defn}[thm]{Definition}

\newtheorem{rmk}[thm]{Remark}
}

\newcommand{\qed}{\hfill \mbox{\raggedright \rule{.07in}{.1in}}}
 
\newenvironment{proof}{\vspace{1ex}\noindent{\bf
Proof}\hspace{0.5em}}{\hfill\qed\vspace{1ex}}
\newenvironment{pfof}[1]{\vspace{1ex}\noindent{\bf Proof of
#1}\hspace{0.5em}}{\hfill\qed\vspace{1ex}}

\newcommand{\cP}{{\mathbb P}}
\newcommand{\R}{{\mathbb R}}
\newcommand{\C}{{\mathbb C}}
\newcommand{\Z}{{\mathbb Z}}
\newcommand{\N}{{\mathbb N}}
\newcommand{\E}{{\mathbb E}}
\newcommand{\cB}{{\mathcal B}}
\newcommand{\cF}{{\mathcal F}}

\newcommand{\bbS}{{{\mathbb S}^{d-1}}}

\newcommand{\Var}{\operatorname{Var}}
\newcommand{\BV}{{\rm BV}}
\newcommand{\sgn}{\operatorname{sgn}}

\newcommand{\infk}{{\SMALL \inf_k}}
\newcommand{\supk}{{\SMALL \sup_k}}
\newcommand{\sumk}{{\SMALL \sum_k}}

\newcommand{\supI}{{\SMALL \sup_I}}
\newcommand{\sumI}{{\SMALL \sum_I}}

\newcommand{\eps}{{\epsilon}}

\newcommand{\SMALL}{\textstyle}

\title{Analytic proof of multivariate stable local large deviations and application to deterministic dynamical systems}

\author{
Ian Melbourne \thanks{Mathematics Institute, University of Warwick, Coventry, CV4 7AL, UK}
\and
Dalia Terhesiu
\thanks{Mathematisch Instituut,
University of Leiden, Niels Bohrweg 1, 2333 CA Leiden, Netherlands}
}

\date{5 September 2020. Updated 13 February 2022.}

\begin{document}

 \maketitle

\begin{abstract}
We give a short analytic proof of local large deviations for i.i.d.\ random variables in the domain of a multivariate $\alpha$-stable law, $\alpha\in(0,1)\cup(1,2]$.  Our method simultaneously covers lattice and nonlattice distributions (and mixtures thereof), bypassing aperiodicity considerations.  The proof applies also to the dynamical setting.

\end{abstract}

\section{Introduction}
\label{sec-intro}

Local large deviation results for i.i.d.\ random variables in the domain of a stable law have been recently obtained by Caravenna and Doney~\cite[Theorem 1.1]{CaravennaDoney} and refined by Berger~\cite[Theorem 2.3]{Berger19}.
We refer to such results as stable local large deviations (stable LLD).

The aim of this paper is three-fold.  First,
we provide a new proof of the stable LLD in Theorem~\ref{thm-lld1d}
(we exclude the case $\alpha=1$ but include the case $\alpha=2$ which was previously omitted).
Second, in Theorem~\ref{thm-lldgen}, we generalise to the multivariate case which for the main part had also been previously omitted.
Our methods bypass aperiodicity considerations and cover lattice and nonlattice distributions simultaneously.
Instead of using Fuk-Nagaev inequalities as was done in~\cite{Berger19,CaravennaDoney}, we give a short analytic proof using Nagaev-type perturbative arguments together with decay of Fourier coefficients.
A major advantage of this approach is that it generalises naturally to the dynamical setting.  This is the third main aim of this paper where in Theorem~\ref{thm-lldGM} we establish the stable LLD for sequences of nonindependent random variables arising from observables of deterministic dynamical systems.

We begin by recalling the scalar i.i.d.\ set up in~\cite{Berger19,CaravennaDoney}.
Let
$X$ be a random variable with $\E X^2=\infty$.  We suppose that 
\begin{equation}\label{eq-rvbt}
\cP(X>x)=(p+o(1))\ell(x) x^{-\alpha},\quad \cP(X \le -x)=(q+o(1))\ell(x) x^{-\alpha}
\end{equation} 
as $x\to\infty$,
where $\alpha\in(0,2]$, $\ell:[0,\infty)\to(0,\infty)$ is slowly varying, and $p,q\ge0$ with $p+q>0$.
Equivalently, $X$ is in the domain of an $\alpha$-stable law $Y_\alpha$ (determined by $\alpha$, $p$, $q$).
Namely, there are sequences $a_n>0$, $b_n\in\R$, such that
\[
\tfrac{1}{a_n}(S_n-b_n)\to_d Y_\alpha.
\]
where $S_n=X_1+\dots +X_n$ and the $X_i$ are independent copies of $X$.

Without loss of generality, we may suppose that $\ell$ is continuous.
Set $\tilde\ell=\ell$ for $\alpha\in(0,2)$ and
$\tilde\ell(x)=1+\int_1^{1+x}\frac{\ell(u)}{u}\, du$ for $\alpha=2$.
Then $a_n$ satisfies 
\[
\lim_{n\to\infty}\frac{n\tilde\ell(a_n)}{a_n^{\alpha}}=1.
\]
Also,
\[
b_n= \begin{cases}  0 & \alpha\in(0,1) \\  
n\E (X 1_{\{|X|\le a_n\}}) & \alpha=1 \\
n\E X & \alpha\in(1,2]
\end{cases}.
\]

Stable local large deviations (the main topic of this paper) concerns estimates for $\cP(S_n\in J)$ for subsets $J\subset \R$ taking into account the location of $J$:

\begin{thm} \label{thm-lld1d} 
Assume~\eqref{eq-rvbt} with $\alpha\in (0,2]$.  
Then for every $h>0$ there is a constant $C>0$ such that 
\begin{align} \label{eq-lld1d}
\cP ( S_n-b_n \in (x- h, x+h] ) \le C\frac{n}{a_n} \,\frac{\tilde\ell(|x|)}{1+|x|^\alpha}
\quad\text{for all $n\ge1$, $x\in\R$.}
\end{align}

In particular, in the lattice case where $X$ is supported on $\Z$,
there is a constant $C>0$ such that
\[
\cP(S_n-[b_n]=N)\le C\frac{n}{a_n} \,\frac{\tilde\ell(|N|)}{1+|N|^\alpha}
\quad\text{for all $n\ge1$, $N\in\Z$.}
\]
\end{thm}

\begin{rmk}
\label{rmk-lldgen}
Caravenna and Doney~\cite[Theorem~1.1]{CaravennaDoney} proved Theorem~\ref{thm-lldgen} for $\alpha\in(0,1)\cup(1,2)$ 
and (amongst other things) this was extended by Berger~\cite[Theorem 2.3]{Berger19}
to the range $\alpha\in(0,2)$ (focusing on the lattice case). 
Our analytic proof covers the range $\alpha\in(0,1)\cup(1,2]$ so the combined results cover the range $\alpha\in(0,2]$.
 Our arguments cover the lattice and nonlattice cases simultaneously.
As mentioned before, our main contribution is to provide a new proof which generalises easily to the dynamical setting.

In the range $|x|\ll a_n$,
the estimate~\eqref{eq-lld1d} follows from the local limit theorems of Gnedenko~\cite[Chap.~9, Sec.~50]{GnedenkoKolmogorov} and Stone~\cite{Stone65}, and in particular the estimate is sharp in the range
$|x|\approx a_n$.
Hence the main content of Theorem~\ref{thm-lld1d} is when $|x|\gg a_n$.
\end{rmk} 

\begin{rmk} We have excluded the problematic case $\alpha=1$ which was completely solved by Berger~\cite{Berger19}.  In fact, our methods apply without modification for $\alpha=1$ in the symmetric case $b_n=0$.
However, in the nonsymmetric case the estimate in Lemma~\ref{lemma-chfacts}(i) below fails (see \cite[Lemma~5]{Erickson70}).  Consequently, without 
refining our methods further we would obtain a suboptimal estimate in Theorem~\ref{thm-lldgen} for $\alpha=1$, $b_n\neq0$.
\end{rmk}

\begin{rmk}
The estimates in Theorem~\ref{thm-lldgen} are proved under assumption~\eqref{eq-rvbt} which is necessary and sufficient for convergence to the stable law $Y_\alpha$.  For stronger estimates under more restrictive hypotheses, we refer 
to~\cite{Berger19,doney,Gouezel11}.
\end{rmk}

Next, we generalise to the multivariate situation.
Let $\bbS = \{x \in \R^d : |x| = 1 \}$ denote the unit sphere in $\R^d$.
(Throughout, $|\;|$ denotes the Euclidean norm.)

\begin{defn} \label{def-reg}
An $\R^d$-valued random variable $X$ is
    \emph{regularly varying} with index $p$ if
    there exists a Borel probability measure $\sigma$
    on
    $\bbS$, such that
    \[
        \lim_{t \to \infty}
        \frac{\cP(|X| > \lambda t, \ X / |X| \in A)}{\cP(|X| > t)}
        = \lambda^p \sigma(A)
    \]
    for all $\lambda > 0$ and all Borel sets $A \subset \bbS$ with $\sigma(\partial A) = 0$.

For $p=-\alpha<0$, we say that $X$ is \emph{nondegenerate} if 
$\int_\bbS|u\cdot\theta|^\alpha\,d\sigma(\theta)>0$ for all $u\in\bbS$.
\end{defn}

Taking $A=\bbS$, we have that $|X|$ is a scalar regularly varying
function.
Hence there exists a slowly varying function $\ell:[0,\infty)\to(0,\infty)$
such that $\cP(|X|>t)=t^p\ell(t)$. 

Let $X$ be an $\R^d$-valued random variable with $\E|X|^2=\infty$.  We suppose that
$X$ is nondegenerate and regularly varying with index $-\alpha$ where $\alpha\in(0,1)\cup(1,2]$.
We define $\tilde\ell$, $a_n$ and $b_n$ as in the scalar case (again taking $\ell$ to be continuous).

Let $Y_\alpha$ denote the $d$-dimensional stable law with spectral measure $\Lambda=\cos\frac{\pi\alpha}{2}\Gamma(1-\alpha)\sigma$ and
characteristic function
\begin{equation} \label{eq-Y}
\E(e^{is\cdot Y_\alpha})=\exp \Big\{-\int_\bbS |s\cdot \theta|^\alpha\big(1-i\sgn(s\cdot\theta)\tan\tfrac{\pi\alpha}{2} \big)\,d\Lambda(\theta) \Big\}, \quad s\in\R^d.
\end{equation}
By~\cite{Rvaceva54}, 
$X$ is in the domain of attraction of $Y_\alpha$.  Indeed,
$a_n^{-1}(S_n-b_n)\to_d Y_\alpha$ as $n\to\infty$ where $S_n=X_1+\dots+X_n$ and the $X_i$ are independent copies of $X$.

Let $\Pi_h(x)=\prod_{j=1}^d(x_j-h,x_j+h]$ for $x\in\R^d$, $h>0$.

\begin{thm} \label{thm-lldgen}
Let $\alpha\in(0,1)\cup(1,2]$.
For every $h>0$, there is a constant $C>0$ such that
\begin{align} \label{eq-lld}
\cP \big( S_n - b_n \in \Pi_h(x)\big) \le C\frac{n}{a_n^d} \, \frac{\tilde\ell(|x|)}{1+|x|^\alpha}
\quad\text{for all $n\ge1$, $x\in\R^d$.}
\end{align}
\end{thm}
Theorem~\ref{thm-lld1d} for $\alpha\neq1$ is immediate from Theorem~\ref{thm-lldgen}.

\begin{rmk} We emphasize that our method works simultaneously for lattice and nonlattice distributions, bypassing any aperiodicity assumptions.
In particular, our result covers distributions that are jointly lattice and nonlattice avoiding the consideration of numerous different cases that arise in the corresponding local limit theorems (see the discussion in~\cite{Doney91}).

Moreover, in the dynamical setting of Section~\ref{sec-dyn}, local limit theorems would require additional hypotheses beyond the probabilistic ones, and our method avoids such hypotheses.
\end{rmk}

\begin{rmk} Berger~\cite{Berger19b} obtains LLD for multivariate stable laws in the case when they are lattice distributed. Furthermore,~\cite{Berger19b} allows the  scalings $a_n$ to vary from component to component. 
\end{rmk}

Our analytic proof of Theorem~\ref{thm-lldgen} is given in Section~\ref{sec-iid}.
In Section~\ref{sec-dyn}, we show that our proof applies to a class of deterministic dynamical systems.
In fact, we prove 
a stronger operator stable LLD in Theorem~\ref{thm-lldGM} which yields the desired stable LLD 
in Corollary~\ref{cor-mixinglld}.

\vspace{-2ex}
\paragraph{Notation}
We write $a_n\ll b_n$ if there are constants $C>0$, $n_0\ge1$ such that
$a_n\le Cb_n$ for all $n\ge n_0$.
As usual, $a_n=o(b_n)$ means that $\lim_{n\to\infty}a_n/b_n=0$
and $a_n\sim b_n$ means that $\lim_{n\to\infty}a_n/b_n=1$.

\section{Stable local large deviations in the i.i.d.\ set up}
\label{sec-iid}

In this section, we provide an analytic proof of Theorem~\ref{thm-lldgen} establishing local large deviations for i.i.d.\ random variables in the domain of a multivariate stable law.
We abbreviate $\Pi_h(0)$ to $\Pi_h$.

Fix $\alpha\in(0,1)\cup(1,2]$, and let $X$ be $\R^d$-valued with $\E|X|^2=\infty$, nondegenerate and regularly varying with index $-\alpha$. 
Define
$a_n$ and $\tilde\ell$ as in Section~\ref{sec-intro}.
Since we exclude the case $\alpha=1$, we can suppose without loss of generality
that $b_n=0$.  This is automatic for $\alpha\in(0,1)$ while for $\alpha\in(1,2]$ we can replace $X$ by $X-\E X$.  In other words, we suppose 
that $\E X=0$ for $\alpha\in(1,2]$.

\subsection{Technical lemmas}

The proof of the results obtained here exploits classical approaches~\cite{GnedenkoKolmogorov,IL} for characteristic functions collected in Lemmas~\ref{lemma-chfact1} and~\ref{lemma-chfacts} below. 

For $s\in\R^d$, define
\(
\Psi(s)=\E(e^{is\cdot X}).
\)

\begin{lemma} \label{lemma-chfact1}
There exist constants $\eps,\,c>0$, such that
\[
|\Psi(s) | \leq \exp\{- c |s|^{\alpha} \tilde\ell(1/|s|)\}
\quad\text{for all $s\in\Pi_{3\eps}$.}
\]
\end{lemma}

\begin{proof}
Since $S_n=X_1+\dots+X_n$ is a sum of i.i.d.\ random vectors and $a_n^{-1}S_n\to_d Y_\alpha$, it follows from
the L\'evy continuity theorem that
$\lim_{n\to\infty}\Psi(s/a_n)^n= \E(e^{is\cdot Y_\alpha})$ for all $s\in\R^d$.
The convergence is uniform on compact sets so
$\lim_{n\to\infty}\Psi(u/a_n)^n= \E(e^{iu\cdot Y_\alpha})$ 
uniformly in $u\in\bbS$.
Hence
\[
\lim_{n\to\infty}|\Psi(u/a_n)|^n=e^{-k_u}
\quad\text{uniformly in $u\in\bbS$,}
\]
where $k_u=\int_\bbS |u\cdot\theta|^\alpha \,d\Lambda(\theta)$.
By compactness, $k_u$ is bounded and hence
\[
\lim_{n\to\infty}n\log|\Psi(u/a_n)|=-k_u
\quad\text{uniformly in $u\in\bbS$.}
\]
Setting $a_n^{-1}=t$ and inverting to obtain
$n\sim(t^\alpha\tilde\ell(1/t))^{-1}$ as $n\to\infty$, we obtain
\[
\lim_{t\to0^+}(t^\alpha\tilde\ell(1/t))^{-1}\log|\Psi(tu)|=-k_u
\quad\text{uniformly in $u\in\bbS$.}
\]
(The details of this last step are identical to the last five lines of the proof of~\cite[Lemma~6.4]{AD01}.)
By nondegeneracy and compactness, $\min_{u\in \bbS}k_u>0$.
Writing $s=tu$,
\[
\log|\Psi(s)|\sim -k_u|s|^\alpha\tilde\ell(1/|s|)
\quad\text{as $s\to0$.}
\]
Hence there exists $\eps>0$ such that 
$\log|\Psi(s)|\le -\frac12 k_u|s|^\alpha\tilde\ell(1/|s|)$
for all $s\in\Pi_{3\eps}$ and
the result follows with $c=\frac12\min_{u\in\bbS}k_u$.
\end{proof}

Throughout this section we fix $\eps$ so that Lemma~\ref{lemma-chfact1} holds.
Let $\partial_j=\partial/\partial {s_j}$.

\begin{lemma} \label{lemma-chfacts}
Let $M>0$. There exists $C>0$ such that
for all $s,\,h\in\Pi_M$,
\begin{itemize}
\item[(i)] $|\Psi(s+h)-\Psi(s)|\le C|h|^\alpha\ell(1/|h|)$ for $\alpha\in(0,1)$.
\item[(ii)] $|\partial_j\Psi(s+h)-\partial_j\Psi(s)|\le C|h|^{\alpha-1}\tilde\ell(1/|h|)$ for $\alpha\in(1,2]$, $j=1,\dots,d$.
\item[(iii)] $|\partial_j\Psi(s)|\le C|s|^{\alpha-1}\tilde\ell(1/|s|)$ for $\alpha\in(1,2]$, $j=1,\dots,d$.
\end{itemize}
\end{lemma}

\begin{proof}
(i) For $K>0$, write
\(
\Psi(s+h)-\Psi(s)=A+B
\)
where
\[
A=\E \big(1_{\{|X|> K\}}(e^{ih\cdot X}-1)e^{is\cdot X}\big), \qquad
B=\E \big(1_{\{|X|\le K\}}(e^{ih\cdot X}-1)e^{is\cdot X}\big).
\]
Note that
\[
|A|\le 2\E 1_{\{|X|> K\}}=2\cP(|X|>K) =2K^{-\alpha}\ell(K).
\]
Next, let $G(x)=\cP(|X|\le x)$ denote the distribution function of $|X|$,
so $1-G(x)=x^{-\alpha}\ell(x)$. Then
\[
|B|\le |h|\E\big(|X| 1_{\{|X|\le K\}}\big)
=|h|\int_0^K x\,dG(x).
\]
Using integration by parts and Karamata's Theorem,
\begin{align*}
\int_0^K x\,dG(x)  & = -\int_0^K x\,d(1-G(x))
=-(x(1-G(x))\Big|_0^K+\int_0^K (1-G(x))\,dx
 \\  & \le \int_0^K (1-G(x))\,dx
=\int_0^K x^{-\alpha}\ell(x)\,dx
=(1-\alpha)^{-1}K^{1-\alpha}\ell(K).
\end{align*}
Hence $|B|\ll |h|K^{1-\alpha}\ell(K)$.
Taking $K\approx 1/|h|$ yields the desired estimate.
\\[.75ex]
(ii)
For $K>0$, write $\partial_j\Psi(s+h)-\partial_j\Psi(s)=i(A+B)$ where
\[
A=\E \big(1_{\{|X|> K\}}X_j(e^{ih\cdot X}-1)e^{is\cdot X}\big), \qquad
B=\E \big(1_{\{|X|\le K\}}X_j(e^{ih\cdot X}-1)e^{is\cdot X}\big).
\]
Using integration by parts and Karamata's Theorem,
\begin{align*}
|A|  & \le  2\E\big(|X|1_{\{|X|>K\}}\big)  =2\int_K^\infty x\,dG(x)
\\ & =2K(1-G(K)) + 2\int_K^\infty (1-G(x))\,dx
=2K^{-(\alpha-1)}\ell(K)+2\int_K^\infty x^{-\alpha}\ell(x)\,dx
\\ & =2\{1+(\alpha-1)^{-1}\}K^{-(\alpha-1)}\ell(K)
\ll K^{-(\alpha-1)}\tilde\ell(K),
\end{align*}
and
\begin{align*}
|B|\le |h|\E\big(|X|^21_{\{|X|\le K\}}\big)
& =|h|\int_0^K x^2\,dG(x)
\le 2|h|\int_0^K x(1-G(x))\,dx
\\ &
=2|h|\int_0^K x^{-(\alpha-1)}\ell(x)\,dx=2(2-\alpha)^{-1}|h|K^{2-\alpha}\tilde\ell(K).
\end{align*}
Again we take $K\approx 1/|h|$.
\\[.75ex]
(iii)  Since $\E X=0$, it follows from part (ii)
that $\partial_j\Psi(s)=\partial_j\Psi(s)-\partial_j\Psi(0)$ satisfies the desired estimate.
\end{proof}

\begin{lemma} \label{lemma-chint}
Let $L:(0,\infty)\to(0,\infty)$ be a continuous slowly varying function.
For all $c>0$, $\beta\ge0$, there exists $C>0$ such that
 for all $n\ge1$,
\[
\int_{\R^d} |s|^\beta L(1/|s|) \exp\{-nc|s|^\alpha\tilde\ell(1/|s|)\}\,ds\le C\frac{L(a_n)}{ a_n^{d+\beta}}.
\]
In particular,
\[
\int_{\Pi_{3\eps}} |s|^\beta L(1/|s|)|\Psi(s)|^n \,ds\le C\frac{L(a_n)}{ a_n^{d+\beta}}.
\]
\end{lemma}

\begin{proof}
Let $I_n=\int_{\R^d} |s|^\beta L(1/|s|) \exp\{-nc|s|^\alpha\tilde\ell(1/|s|)\}\,ds$.
Using the change of variables $s=\sigma/a_n$, we obtain
   \[
I_n=\frac{L(a_n)}{a_n^{d+\beta}}J_n \quad\text{where} \quad
   J_n=\int_{\R^d} \frac{L (a_n/|\sigma|)}{L(a_n)} |\sigma|^\beta \exp\{-c n|\sigma|^\alpha a_n^{-\alpha} \tilde\ell(a_n/|\sigma|)\} \, d \sigma.
\]
By Potter's bounds, for any $\delta\in(0,\alpha)$, 
there exists $c'>0$ such that
\[
J_n \ll \int_{|\sigma|>1} |\sigma|^{\beta+\delta} \exp\{-c' n|\sigma|^{\alpha-\delta} a_n^{-\alpha} \tilde\ell(a_n)\} \, d \sigma.
\]
Recalling that $na_n^{-\alpha}\tilde\ell(a_n) \sim 1$,
there exists $c''>0$ such that
\[
 J_n\ll \int_{|\sigma|>1} |\sigma|^{\beta+\delta} e^{-c'' |\sigma|^{\alpha-\delta} } \, d \sigma<\infty.
\]
A similar argument deals with the integral on $\{|\sigma|\le 1\}$.

The final statement follows from Lemma~\ref{lemma-chfact1}.
\end{proof}

\subsection{Proof of the stable LLD}
\label{subsec-stprlld}

In this subsection, we prove Theorem~\ref{thm-lldgen}.
We suppose without loss of generality that $h=1$; the result for smaller cubes is immediate and the result for larger cubes can be obtained by taking unions of smaller cubes.

As in~\cite{Stone65}, we convolve with a suitable function $\gamma$
with compactly supported Fourier transform $\hat\gamma$.
Specifically, fix a continuous integrable function $\gamma_0:\R\to[0,\infty)$ with $\gamma_0\ge1$ on $[-2,2]$ such that its Fourier transform $\widehat{\gamma_0}$ is real-valued, even and $C^2$ with support in $[-\eps,\eps]$.\footnote{
It is easily verified that such a $\gamma_0$ exists.  Start with
an even $C^\infty$ function $\widehat{\gamma_0}:\R\to[0,\infty)$ supported in $[-\frac{\eps}{2},\frac{\eps}{2}]$
with inverse Fourier transform $\gamma_0$.  Then $\gamma_0$ is real-valued and $C^\infty$.
Taking $\widehat{\gamma_0}\not\equiv0$ ensures that $\gamma_0(0)=\frac{1}{2\pi}\int_{-\infty}^\infty \widehat{\gamma_0}(\xi)\,d\xi>0$.   Replacing $\gamma_0(x)$ by $\gamma_0(ax)$ with $a$ sufficiently small ensures that $\gamma_0>0$ on $[-2,2]$.  (Such a scaling shrinks the support of $\widehat{\gamma_0}$, so the new $\widehat{\gamma_0}$ remains supported in $-\frac{\eps}{2},\frac{\eps}{2}]$.)  Next, replace $\gamma_0$ by $c\gamma_0$ for $c$ sufficiently large,
ensuring that $\gamma_0\ge1$ on $[-2,2]$.
Finally, replace $\gamma_0$ by $\gamma_0^2$ and $\widehat{\gamma_0}$ by $\widehat{\gamma_0}\star\widehat{\gamma_0}$ to ensure that $\gamma_0\ge0$.
}
For $s\in\R$, define
\[
r_0(s)=\frac{1}{2\pi} \frac{\sin s}{s}\widehat{\gamma_0}(s).
\]
We note that $r_0$ is $C^2$ and supported in $[-\eps,\eps]$.
Define $\gamma:\R^d\to[0,\infty)$ and $r:\R^d\to\R$,
\[
\gamma(y)=\gamma_0(y_1)\cdots \gamma_0(y_d), \qquad
r(s)=r_0(s_1)\cdots r_0(s_d).
\]

\begin{lemma} \label{lemma-iidgamma}
For $n\ge1$, $x\in\R^d$,
\[
\cP\big(S_n\in \Pi_1(x)\big)\le
\int_{\R^d} e^{-is\cdot x}r(s) \Psi(s)^n \, ds.
\]
\end{lemma}

\begin{proof}
By the Fourier inversion formula,
\begin{align} \label{eq-inv}
\gamma(y)=\frac{1}{(2\pi)^d}\int_{\Pi_\eps} e^{is\cdot y} \hat\gamma(s)\,ds
=\frac{1}{(2\pi)^d}\int_{\Pi_\eps} e^{-is\cdot y} \hat\gamma(s)\,ds.
\end{align}
Let $F_n(x)=\cP \big(S_n\in \prod_{j=1}^d (-\infty,x_j]\big)$.
By Fubini,
\begin{align}  \nonumber
 \int_{\Pi_1(x)}\int_{\R^d} & \gamma(y-y')dF_n(y') \,dy
\\ & =
\frac{1}{(2\pi)^d}\int_{\Pi_\eps} 
\Big(\int_{\Pi_1(x)} e^{-is\cdot y}\,dy\Big)\hat\gamma(s)\Big(\int_{\R^d} e^{is\cdot y'}\,dF_n(y') \Big) \,ds
\label{eq-gamma} 
 \\ & =
2^d\int_{\Pi_\eps} e^{-is\cdot x}r(s) \Psi(s)^n \,ds.
\nonumber
\end{align}

Next, since $\gamma_0\ge0$ and $\gamma_0|_{[-2,2]}\ge1$,
\[
\int_{\R^d} \gamma(y-y')\,dF_n(y')\ge 
 \int_{\Pi_2(y)} \,dF_n(y')
=\cP\big(S_n\in \Pi_2(y)\big).
\]
Hence
\begin{align*}
\int_{\Pi_1(x)}
\int_{\R^d} \gamma(y-y')\,dF_n(y')\,dy
& \ge
\int_{\Pi_1(x)} \cP\big(S_n\in \Pi_2(y)\big)\,dy
 \ge 2^d\,\cP\big(S_n\in \Pi_1(x)\big).
\end{align*}
Combining this with~\eqref{eq-gamma} yields the desired result.
\end{proof}

\begin{pfof}{Theorem~\ref{thm-lldgen}}
Recall that we have reduced to the case $h=1$.
By Lemma~\ref{lemma-iidgamma}, it suffices to estimate
$I_{n,x}=\int_{\R^d} e^{-is\cdot x}r(s) \Psi(s)^n \,ds$.
Since $r$ is bounded and supported in $\Pi_\eps$, it follows from Lemma~\ref{lemma-chint} that
\[
|I_{n,x}|\ll \int_{\Pi_\eps}|\Psi(s)|^n\,ds\ll a_n^{-d}.
\]

If $n\gg (1+|x|^\alpha)/\tilde\ell(|x|)$, then
\(
|I_{n,x}|\ll a_n^{-d}\ll \frac{n}{a_n^d}\,\frac{\tilde\ell(|x|)}{1+|x|^\alpha}
\)
as required.
Hence to complete the proof it suffices to consider the case
$n\le c (1+|x|^\alpha)/\tilde\ell(|x|)$ for some $c>0$ fixed.
In particular, we can suppose that $|x|\ge \pi/\eps$, and it suffices to prove
that 
\(
|I_{n,x}|\ll \frac{n}{a_n^d}\frac{\tilde\ell(|x|)}{|x|^\alpha}
\)
for $n\le c'|x|^\alpha/\tilde\ell(|x|)$.
Since $a_n^\alpha/\tilde\ell(a_n)\sim n$ and $y\mapsto y^\alpha/\tilde\ell(y)$
is asymptotically increasing, this last restriction can be written as $a_n\le |x|$.

\vspace{2ex}
\noindent {\bf{The case $\alpha\in (0,1)$.}}
We exploit the modulus of continuity of $\Psi$ (see, for instance,~\cite[Chapter 1]{Katzn}).   Note that
\(
I_{n,x}=-\int_{\R^d} e^{-is\cdot x}r(s-h)\Psi(s-h)^n\,ds,
\)
where $h=\pi x/|x|^2$.
Hence
\begin{align} \label{eq-mod}
|I_{n,x}|=\frac12\Big|\int_{\R^d} e^{-is\cdot x}\big(r(s)\Psi(s)^n-r(s-h)\Psi(s-h)^n\big)\,ds\Big|
\le  I_1+I_2
\end{align}
where
\[
I_1= \int_{\R^d} |r(s)-r(s-h)||\Psi(s)|^n\,ds,
\quad
I_2=\int_{\R^d} |r(s-h)||\Psi(s)^n-\Psi(s-h)^n|\,ds.
\]

Since $r$ is supported in $\Pi_\eps$ and
$|x|\ge\pi/\eps$, the integrands in $I_1$ and $I_2$ are supported in $\Pi_{2\eps}$.
Using also that $r$ is bounded and Lipschitz,
\[
I_1\ll |x|^{-1} \int_{\Pi_{2\eps}} |\Psi(s)|^n\,ds,
\quad
I_2\ll \int_{\Pi_{2\eps}} |\Psi(s)^n-\Psi(s-h)^n|\,ds.
\]
By Lemma~\ref{lemma-chint},
\(
I_1\ll \frac{1}{a_n^d} \frac{1}{|x|}.
\)

Next recall the inequality
\begin{equation}
\label{eq-est}
|u^n-v^n|\le n|u-v| (|u|^{n-1}+|v|^{n-1}),
\end{equation}
which holds for all $u,v\in\C$, $n\ge1$.
Using this and Lemma~\ref{lemma-chfacts}(i),
\[
|\Psi(s)^n-\Psi(s-h)^n|\ll n |x|^{-\alpha}\ell(|x|)\big(|\Psi(s)|^{n-1}+|\Psi(s-h)|^{n-1}\big).
\]
Hence by Lemma~\ref{lemma-chint},
\(
I_2\ll n|x|^{-\alpha}\ell(|x|) \int_{\Pi_{3\eps}} |\Psi(s)|^{n-1}\,ds
\ll \frac{n}{a_n^d} \frac{\ell(|x|)}{|x|^\alpha}.
\)

\vspace{2ex}
\noindent {\bf{The case $\alpha\in (1,2]$.}}
Choose $j$ so that $|x_j|=\max\{|x_1|,\dots,|x_d|\}$.
Let $D_n=\Psi^{n-1}\partial_j\Psi$.
Integrating by parts, $I_{n,x}=E_1+E_2$ where
\[
E_1=\frac{1}{ix_j} \int_{\R^d} e^{-is\cdot x}\partial_jr(s)\Psi(s)^n\,ds,
\qquad
E_2=\frac{n}{ix_j} \int_{\R^d} e^{-is\cdot x}r(s)D_n(s)\,ds.
\]
Integrating by parts once more, and using that $r$ is $C^2$ and supported in $\Pi_\eps$,
\begin{align*}
|E_1| & \le \frac{1}{x_j^2}\int_{\Pi_\eps} |\partial_j^2 r(s)||\Psi(s)|^n\,ds
+
\frac{n}{x_j^2}\int_{\Pi_\eps} |\partial_jr(s)||D_n(s)|\,ds
\\
& \ll \frac{1}{|x|^2}\int_{\Pi_\eps} |\Psi(s)|^n\,ds
+
\frac{n}{|x|^2}\int_{\Pi_\eps} |\Psi(s)|^{n-1}\,ds.
\end{align*}
(Here, we used also that $\partial_j\Psi$ is bounded on $\Pi_\eps$ by Lemma~\ref{lemma-chfacts}(iii).)
By Lemma~\ref{lemma-chint},
\[
|E_1| \ll \frac{n}{a_n^d}\frac{1}{|x|^2}
\ll \frac{n}{a_n^d}\frac{\tilde\ell(|x|)}{|x|^\alpha}.
\]

Next, we exploit the modulus of continuity of $rD_n$, writing $h=\pi x_j^{-1}e_j$ (where $e_j\in\R^d$ is the $j$'th unit vector) and
\begin{align*}
|E_2|
& \le \frac{n}{|x_j|}\int_{\R^d} |r(s)-r(s-h)|\,|D_n(s)|\,ds
+\frac{n}{|x_j|}\int_{\R^d}  |r(s-h)|\,|D_n(s)-D_n(s-h)|\,ds
\\ & \ll \frac{n}{|x|^2}\int_{\Pi_{2\eps}}  |\Psi(s)|^{n-1}\,ds
+\frac{n}{|x|}\int_{\Pi_{2\eps}}  |D_n(s)-D_n(s-h)|\,ds.
\end{align*}
Again,
\(
\frac{n}{|x|^2}\int_{\Pi_{2\eps}}  |\Psi(s)|^{n-1}\,ds\ll \frac{n}{a_n^d}\frac{1}{|x|^2}
\ll \frac{n}{a_n^d}\frac{\tilde\ell(|x|)}{|x|^\alpha},
\)
so it remains to estimate
\[
J=\frac{n}{|x|}\int_{\Pi_{2\eps}}  |D_n(s)-D_n(s-h)|\,ds.
\]

Relabel $\{s,s-h\}=\{s_1,s_2\}$ where $|\Psi(s_1)|\le |\Psi(s_2)|$.
Then
$J=J_1+J_2$
where $J_i=\frac{n}{|x|}\int K_i$ for $i=1,2$, and
\[
K_1= |\Psi(s_1)|^{n-1}|\partial_j\Psi(s_1)-\partial_j\Psi(s_2)|,
\qquad
K_2 =|\Psi(s_1)^{n-1}-\Psi(s_2)^{n-1}||\partial_j\Psi(s_2)|.
\]
By Lemma~\ref{lemma-chfacts}(ii),
\[
K_1 \ll 
|x_j|^{1-\alpha}\tilde\ell(|x_j|)|\Psi(s_1)|^{n-1}
\ll |x|^{1-\alpha}\tilde\ell(|x|)|\Psi(s_1)|^{n-1}.
\]
Hence by Lemma~\ref{lemma-chint},
\[
J_1\ll n\frac{\tilde\ell(|x|)}{|x|^\alpha}
\int_{\Pi_{3\eps}} |\Psi(s)^{n-1}|\,ds
\ll
\frac{n}{a_n^d} \frac{\tilde\ell(|x|)}{|x|^\alpha}.
\]

Next, by~\eqref{eq-est},
\begin{align*}
K_2
 \ll n|\partial_j\Psi(s_2)||\Psi(s_1)-\Psi(s_2)| |\Psi(s_2)|^{n-2}.
\end{align*}
By the mean value theorem for vector-valued functions,
there exists $s^*$ between $s_1$ and $s_2$ such that
\[
K_2 \ll n|x_j|^{-1}|\partial_j\Psi(s_2)| |\partial_j\Psi(s^*)|
|\Psi(s_2)|^{n-2} \ll n|x|^{-1}( K_3+K_4 )
\]
where
\[
K_3  = |\partial_j\Psi(s_2)||\partial_j\Psi(s_2)-\partial_j\Psi(s^*)| 
|\Psi(s_2)|^{n-2},  \quad 
K_4  = |\partial_j\Psi(s_2)|^2 |\Psi(s_2)|^{n-2}.
\]
Correspondingly, we have $J_2\ll J_3+J_4=n^2|x|^{-2}(\int K_3+\int K_4)$.

By Lemma~\ref{lemma-chfacts}(ii),(iii),
\(
K_3\ll |x|^{1-\alpha}
\tilde\ell(|x|)|s_2|^{\alpha-1}\tilde\ell(1/|s_2|) |\Psi(s_2)|^{n-2}.
\)
Hence, by Lemma~\ref{lemma-chint},
\begin{align*}
J_3 & \ll n^2|x|^{-(\alpha+1)}\tilde\ell(|x|)\int_{\Pi_{3\eps}} |s|^{\alpha-1}\tilde\ell(1/|s|)
|\Psi(s)|^{n-2}
\,ds
\ll n^2|x|^{-(\alpha+1)}\tilde\ell(|x|) a_n^{-(d+\alpha-1)}\tilde\ell(a_n) 
\\ &  \sim n|x|^{-(\alpha+1)}\tilde\ell(|x|) a_n^{-(d-1)}
 =\frac{n}{a_n^d}\frac{\tilde\ell(|x|)}{|x|^\alpha} \frac{a_n}{|x|}
 \ll\frac{n}{a_n^d}\frac{\tilde\ell(|x|)}{|x|^\alpha},
\end{align*}
where we have used that $a_n\le |x|$.

Finally, by Lemma~\ref{lemma-chfacts}(iii),
\(
K_4  
\ll |s_2|^{2(\alpha-1)}\tilde\ell(1/|s_2|)^2 |\Psi(s_2)|^{n-2}.
\)
Hence, by Lemma~\ref{lemma-chint},
\begin{align*}
J_4
& \ll n^2|x|^{-2}\int_{\Pi_{3\eps}} |s|^{2(\alpha-1)}\tilde\ell(1/|s|)^2 |\Psi(s)|^{n-2}\,ds
 \ll n^2|x|^{-2} a_n^{-d}a_n^{-2(\alpha-1)}\tilde\ell(a_n)^2
\\ &  \sim n|x|^{-2} a_n^{-d}a_n^{2-\alpha}\tilde\ell(a_n)
= \frac{n}{a_n^d} \frac{\tilde\ell(|x|)}{|x|^\alpha}
\frac{a_n^{2-\alpha}\tilde\ell(a_n)}{|x|^{2-\alpha}\tilde\ell(|x|)}
\ll  \frac{n}{a_n^d} \frac{\tilde\ell(|x|)}{|x|^\alpha}.
\end{align*}
Combining the estimates for $J_1$, $J_3$, $J_4$
we obtain that $J\ll \frac{n}{a_n^d} \frac{\tilde\ell(|x|)}{|x|^\alpha}$ completing the proof.
\end{pfof}

\section{Stable LLD for dynamical systems }
\label{sec-dyn}

In this section we show that the previous results can be generalized to a  class of deterministic dynamical systems.

The dynamical systems considered here are obtained by iterating
a measure-preserving map $f:\Lambda\to \Lambda$ 
on a probability space $(\Lambda,\mu)$.
Starting with an initial condition $z$ distributed according to $\mu$, the process $z$, $f z$, $f^2 z=f(fz),\,\ldots$ on $\Lambda^\N$ is stationary,
with distribution $\mu\otimes \delta_{fz}\otimes \delta_{f^2z}\otimes\cdots$.
Also, if $v:\Lambda\to\R^d$ is a measurable observable, then the $\R^d$-valued process $v$, $v\circ f$, $v\circ f^2,\, \ldots$ is stationary.
The stationary distribution $\mu\otimes \delta_{fz}\otimes \delta_{f^2z}\otimes\cdots$ should be compared with the distribution $\mu\otimes\mu\otimes\mu\otimes\cdots$ in the i.i.d.\ case and explains the word ``deterministic'': once the initial condition is specified the future dynamics is uniquely specified with no further randomness. This contrasts with standard probabilistic settings, where fresh randomness is typically injected at each time step. This means that techniques from probability theory have to be reinforced with methods from ``smooth ergodic theory'' with suitable regularity conditions imposed on the map $f$, the measure $\mu$ and the observable $v$.

The abstract setting in this section includes two classes of deterministic dynamical systems: Gibbs-Markov maps with $v$ piecewise H\"older, and more generally AFU maps with $v$ piecewise bounded variation.
(The Gauss map $f(z)=z^{-1}-[z^{-1}]$, $z\in(0,1]$ is a classical example of a Gibbs-Markov map.)
Define $v_n=\sum_{k=0}^{n-1}v\circ f^k$. For $v$ in the domain of an $\alpha$-stable law $Y_\alpha$, $\alpha\in(0,2]$,  it is shown in~\cite{AD01,ADb} that $v_n$ suitably normalised converges in distribution to $Y_\alpha$.
Here, we prove the corresponding local large deviation estimates.

In Subsection~\ref{sec-dynset}, we state our main results for deterministic dynamical systems in an abstract functional-analytic framework.
Subsections~\ref{sec-tech} and~\ref{sec-pfoplld} contain the proof of these results.
In Subsections~\ref{sec-GM} and~\ref{sec-AFU}, we verify that Gibbs-Markov maps and AFU maps are covered by these results.

\subsection{Dynamical systems set up}
\label{sec-dynset}

Let $f:\Lambda\to \Lambda$ be a measure-preserving map
on a probability space $(\Lambda,\mu)$.
Let $v:\Lambda\to\R^d$ be a measurable observation with $\int_\Lambda |v|^2\,d\mu=\infty$.
We fix $\alpha\in(0,1)\cup(1,2]$ throughout and
assume
\begin{itemize}
\item[(H1)] $v$ is nondegenerate and regularly varying with index $-\alpha$ as in 
Definition~\ref{def-reg}.
\end{itemize}
Define $\ell$, $\tilde\ell$ and $a_n$ as in Section~\ref{sec-intro}.
As in the i.i.d.\ case, we set $b_n=0$ for $\alpha<1$ and $b_n=n
\int_\Lambda v\,d\mu$ for $\alpha>1$.

Let $R:L^1\to L^1$ be the transfer operator for $f$ defined via the formula
$\int_{\Lambda} R\phi\,\psi\,d\mu = \int_{\Lambda}\phi\,\psi\circ f\,d\mu$.
Given $s\in\R^d$, define the perturbed operator $R(s):L^1\to L^1$ by $R(s)\phi=R(e^{is\cdot v}\phi)$. 

We assume that there is a Banach
space $\cB\subset L^\infty$ containing constant
functions, with norm $\|\;\|$ satisfying $|\phi|_\infty\le \|\phi\|$ for
$\phi\in\cB$, such that
\begin{itemize}
\item[(H2)] There exist $\eps>0$, $C>0$ such that 
for all $s,h\in \Pi_\eps$, $j=1,\dots,d$,
\begin{itemize}
\item[(i)] 
$\|R(s)\|\le C$ for $\alpha\in(0,2]$ and
$\|\partial_jR(s)\|\le C$ for $\alpha\in(1,2]$.
\item[(ii)] 
 $\|R(s+h)-R(s)\|\le C|h|^{\alpha}\ell(1/|h|)$ for $\alpha\in(0,1)$ and
$\|R( s+h)-R(s)\|\le C|h|$ for $\alpha\in(1,2]$.
\item[(iii)] 
 $\|\partial_jR(s+h)-\partial_jR(s)\|\le C|h|^{\alpha-1}\tilde\ell(1/|h|)$ for $\alpha\in(1,2]$.
\end{itemize}
\end{itemize}

 Since $R(0)=R$ and $\cB$
contains constant functions, $1$ is an eigenvalue of $R(0)$. 
We assume:
\begin{itemize}\item[(H3)]
The eigenvalue $1$ is simple,
and the remainder of the spectrum of $R(0):\cB\to\cB$ is contained in a disk of radius less than $1$.
\end{itemize}
Let $v_n=\sum_{k=0}^{n-1} v\circ f^k$.
Under these hypotheses, we have distributional convergence 
\(
 (v_n-b_n)/a_n\to_d Y_\alpha
\)
to the $d$-dimensional $\alpha$-stable law $Y_\alpha$ in~\eqref{eq-Y} by~\cite{AD01}.

\begin{rmk} Under additional aperiodicity assumptions, local limit theorems are proved in various situations in~\cite{AD01,ADb,MT20}. Our results do not require aperiodicity assumptions, so we do not discuss these issues further.
\end{rmk}

We can now state the main result in the dynamical setting: namely an operator stable LLD.  

\begin{thm} \label{thm-lldGM} 
Let $\alpha\in(0,1)\cup(1,2]$ and assume (H1)--(H3). 
Then there exists a constant $C>0$ such that
\[ |R^n 1_{\{v_n-b_n\in\Pi_1(x)\}}|_\infty
\le C  \dfrac{n}{a_n^d}\, \dfrac{\tilde\ell(|x|)}{1+|x|^\alpha}
\quad\text{for all $n\ge1$, $|x|\in\R^d$}.
\]
\end{thm}

 A consequence of Theorem~\ref{thm-lldGM} is the usual stable LLD.
\begin{cor}
\label{cor-mixinglld} 
Let $\alpha\in(0,1)\cup(1,2]$ and assume (H1)--(H3).
Then for every $h>0$, there exists $C>0$ such that
\[
\Big|\int_\Lambda \phi\,\psi\circ f^n\,1_{\{v_n-b_n \in \Pi_h(x) \}}\,d\mu\Big|\le C|\phi|_\infty \,|\psi|_1\frac{n}{a_n^d} \, \frac{\tilde\ell(|x|)}{1+|x|^\alpha}
\]
for all $n\ge1$, $x\in\R^d$, $\phi\in L^\infty$, $\psi\in L^1$.

In particular, taking $\phi=\psi=1$, we obtain that
\[
\mu\{v_n-b_n \in \Pi_h(x) \}\le C\frac{n}{a_n^d} \, \frac{\tilde\ell(|x|)}{1+|x|^\alpha}
\quad\text{for all $n\ge1$, $x\in\R^d$.}
\]
\end{cor}

\begin{proof}
As in Section~\ref{subsec-stprlld},
we can suppose without loss of generality that $h=1$.  
Now,
\[
\int_\Lambda \phi\,\psi\circ f^n\,1_{\{v_n-b_n \in \Pi_1(x) \}}\,d\mu
=\int_\Lambda \psi R^n(1_{\{v_n-b_n\in \Pi_1(x)\}}\phi)\,d\mu,
\]
so by positivity of the operator $R$,
\begin{align*}
\Big|\int_\Lambda \phi\,\psi\circ f^n\, &1_{\{v_n-b_n \in \Pi_1(x) \}} \,d\mu\Big|
 \le |R^n(1_{\{v_n-b_n\in \Pi_1(x)\}}\phi)|_\infty |\psi|_1
\\ &  \le |R^n(1_{\{v_n-b_n\in \Pi_1(x)\}}|\phi|_\infty)|_\infty |\psi|_1
= |R^n 1_{\{v_n-b_n\in \Pi_1(x)\}}|_\infty |\phi|_\infty|\psi|_1.
\end{align*}
The result follows by Theorem~\ref{thm-lldGM}.
\end{proof}

%
%
%
The proof of Theorem~\ref{thm-lldGM} takes up the remainder of this section.
We suppose from now on that (H1)--(H3) hold.
As in Section~\ref{sec-iid},
we can suppose without loss of generality that $b_n=0$.
Equivalently,
for $\alpha\in(1,2]$ we can suppose that $\int_\Lambda v\,d\mu=0$

\subsection{Technical lemmas}
\label{sec-tech}

By (H2) and (H3), there exists $\eps>0$ and a continuous family $\lambda(s)$ of simple
eigenvalues of $R(s)$ for $s\in\Pi_{3\eps}$ with
$\lambda(0)=1$.  
The associated spectral projections $P(s)$, $s\in\Pi_{3\eps}$, form a continuous family of bounded linear operators on $\cB$.
Moreover, there is a continuous family of linear operators $Q(s)$ on $\cB$
and constants $C>0$, $\delta_0\in(0,1)$ such that
\begin{align}
\label{eq-sp}
 & R(s)=\lambda(s) P(s)+Q(s)
\quad\text{for $s\in\Pi_{3\eps}$}.
\\
\label{eq-sp2}
 & \|Q(s)^n\|\le C\delta_0^n
\quad\text{for $s\in\Pi_{3\eps}$, $n\ge1$}.
\end{align}
Hence we can shrink $\eps$ so that
\begin{align}
\label{eq-sp3}
\|R(s)^n\|\le C|\lambda(s)|^n
\quad\text{for $s\in\Pi_{3\eps}$, $n\ge1$}.
\end{align}

Let $\zeta(s)=\frac{P(s) 1}{\int P(s) 1\, d\mu}$ be the normalized eigenvector corresponding to $\lambda(s)$.

\begin{lemma} \label{lemma-PQ}
There exists $\eps>0$ such that
the properties of $R(s)$ listed in (H2) are inherited by
$P(s)$, $Q(s)$, $\lambda(s)$ and $\zeta(s)$
for all $s,h\in\Pi_{3\eps}$.
\end{lemma}

\begin{proof} This is a standard consequence of perturbation theory for smooth families of operators. 
\end{proof}

The next result is the analogue of Lemmas~\ref{lemma-chfact1} and~\ref{lemma-chfacts}(iii) for  $\lambda(s)$.
 
\begin{lemma}
\label{lemma-eigv} 
There exist constants $\eps$, $c$, $C>0$ such that
the following hold for all 
\mbox{$s\in\Pi_{3\eps}$},
\begin{itemize}
 \item[(i)] 
$ |\lambda(s)|  \leq \exp\{- c |s|^{\alpha} \tilde\ell(1/|s|)\}$
for $\alpha\in(0,1)\cup(1,2]$.

\item[(ii)] $|\partial_j\lambda(s)|\le C |s|^{\alpha-1}\tilde\ell(|1/|s|)$
 for $\alpha\in(1,2]$, $j=1,\dots,d$.
\end{itemize}
\end{lemma}

\begin{proof} 
(i) Write
\begin{align*} 
\lambda(s) & =\int R(s)\zeta(s)\,d\mu=
\int R(s)1\,d\mu+\int(R(s)-R(0))(\zeta(s)-\zeta(0))\,d\mu 
\\ & =
\int_{\Lambda} e^{is\cdot v}\,  d\mu+\int_{\Lambda} (R(s)-R(0))(\zeta(s)-\zeta(0))\, d\mu=\Psi(s)+V(s).
\end{align*}
The estimate for $\Psi$ in
Lemma~\ref{lemma-chfact1}
is unchanged (since the distribution of $v$ is given by (H1))
so it suffices to verify that the contributions from $V$ are negligible. 

For $\alpha\in(0,1)$, we choose $\alpha'\in(\frac12\alpha,\alpha)$.
Then $\|R(s)-R(0)\|\ll |s|^{\alpha'}$ by (H2) and
$\|\zeta(s)-\zeta(0)\|\ll |s|^{\alpha'}$ by Lemma~\ref{lemma-PQ}.
Since $\cB\subset L^\infty$,
\[
|V(s)|\ll \|R(s)-R(0)\|\|\zeta(s)-\zeta(0)\|
\ll |s|^{2\alpha'}=o(|s|^\alpha\ell(1/|s|)).
\]
Similarly, $|V(s)|\ll |s|^2=o(|s|^\alpha\tilde\ell(1/|s|))$ when $\alpha\in(1,2]$.
This completes the proof of part~(i).
\\[.75ex]
(ii) 
By the formula in part (i), $\partial_j\lambda(0)=i\int_\Lambda v_j\,d\mu=0$. Hence $\partial_j\lambda(s)=\partial_j\lambda(s)-\partial_j\lambda(0)$ so the estimate follows from
Lemma~\ref{lemma-PQ}.
\end{proof} 

From now on, $\eps>0$ is fixed in accordance with the above properties.

\begin{cor} \label{cor-eigvint}
Let $L:(0,\infty)\to(0,\infty)$ be a continuous slowly varying function.
For all $c>0$, $\beta\ge0$, there exists $C>0$ such that
 for all $n\ge1$,
\[
\int_{\Pi_{3\eps}} |s|^\beta L(1/|s|)|\lambda(s)|^n \,ds\le C\frac{L(a_n)}{ a_n^{d+\beta}}.
\]
\end{cor}

\begin{proof}
This follows from Lemmas~\ref{lemma-chint} and~\ref{lemma-eigv}(i).
\end{proof}

We require the following estimates on 
the derivatives of $R(s)^n$ and $Q(s)^n$.

\begin{lemma}\label{lemma-op} 
Let $\alpha\in(1,2]$ and fix $\delta_1\in (\delta_0, 1)$.  Then
there exists $C>0$ such that 
for all $s,s+h\in\Pi_{3\eps}$, $j=1,\dots,d$,
\[
\|\partial_j(R(s)^n)\|\le C n|\lambda(s)|^{n-1}
\quad\text{and}\quad
\|\partial_j(Q(s+h)^{n})-\partial_j(Q(s)^{n})\|\le C\delta_1^n |h|^{\alpha-1}\tilde\ell(1/|h|).
\]
\end{lemma}

\begin{proof}
We start from
\(
\partial_j(R(s)^n)=\sum_{k=0}^{n-1}R(s)^k\partial_jR(s)R(s)^{n-k-1}.
\)
By~\eqref{eq-sp3} and (H2)(i)
\[
\|\partial_j(R(s)^n)\|
\ll \sum_{k=0}^{n-1}|\lambda(s)|^k|\lambda(s)|^{n-k-1}
=n|\lambda(s)|^{n-1}.
\]

Next, fix $\delta_2\in(\delta_0,\delta_1)$.
By~\eqref{eq-sp2} and Lemma~\ref{lemma-PQ},
\begin{align} \label{eq-op} \nonumber
\|Q(s+h)^n-Q(s)^n\| & \le \sum_{k=0}^{n-1} \|Q(s+h)^k\|\,\|Q(s+h)-Q(s)\|\,\|Q(s)^{n-k-1}\|
\\ & \ll |h|\sum_{k=0}^{n-1}\delta_0^n \ll \delta_2^n|h|.
\end{align}
Let $k,m \ge0$ with $k+m =n-1$.  Then
\begin{align*}
(Q^k\partial_jQQ^m )(s+h)  -(Q^k\partial_jQQ^m )(s)
& =\big(Q(s+h)^k-Q(s)^k\big)\partial_jQ(s+h)Q(s+h)^m 
\\  & \qquad +Q(s)^k\big(\partial_jQ(s+h)-\partial_jQ(s)\big)Q(s+h)^m 
\\ & \qquad +Q(s)^k\partial_jQ(s)\big(Q(s+h)^m -Q(s)^m \big)
\end{align*}
so by~\eqref{eq-sp2},~\eqref{eq-op} and Lemma~\ref{lemma-PQ},
\[
\|(Q^k\partial_jQQ^m )(s+h)  -(Q^k\partial_jQQ^m )(s)\|  \ll \delta_2^{n-1}|h|+
\delta_2^{n-1}|h|^{\alpha-1}\tilde\ell(1/|h|)
\ll \delta_2^n|h|^{\alpha-1}\tilde\ell(1/|h|).
\]
Substituting into $\partial_j(Q(s)^{n})=\sum_{k=0}^{n-1} (Q^k\partial_jQQ^{n-k-1})(s)$
we obtain
$\|\partial_j(Q(s+h)^{n})- (\partial_jQ(s)^{n})\|\ll 
n\delta_2^n|h|^{\alpha-1}\tilde\ell(1/|h|)
\ll \delta_1^n|h|^{\alpha-1}\tilde\ell(1/|h|)$ as required.
\end{proof}

\begin{cor} \label{cor-Rn}
(i) Let $\alpha\in(0,1)\cup(1,2]$.
There exists $C>0$ such that for all $|h|\le\eps$,
\[
\int_{\Pi_{2\eps}} \|R(s)^k\|\|R(s+h)^m \|\,ds  \le C a_n^{-d}
\quad\text{for all $k,m \ge0$, $n\ge1$ with $k+m =n$}.
\]

\vspace{1ex} \noindent
(ii) Let $\alpha\in(1,2]$.
There exists $C>0$ such that 
\[
\int_{\Pi_{2\eps}} \|\partial_j(R(s)^n)\|\,ds  \le C na_n^{-d}
\quad\text{for all $n\ge1$, $j=1,\dots,d$}.
\]
\end{cor}

\begin{proof}
(i) 
By~\eqref{eq-sp3},
\[
\|R(s)^k\|\,\|R(s+h)^m \|  \ll |\lambda(s)|^k|\lambda(s+h)|^m \ll
 |\lambda(s)|^n+|\lambda(s+h)|^n.
\]
Also, by Lemma~\ref{lemma-op},
\(
\|\partial_j(R(s)^n)\| \ll n |\lambda(s)|^{n-1}.
\)
Hence both parts follow from Corollary~\ref{cor-eigvint},
\end{proof}

\subsection{Proof of the operator stable LLD}
\label{sec-pfoplld}

In this subsection, we prove Theorem~\ref{thm-lldGM}.
Define $r:\R\to\R$ as in Section~\ref{subsec-stprlld}.
Recall that $r$ is $C^2$, even, and supported in $\Pi_\eps$.

\begin{lemma} \label{lemma-gamma}
\(
1_{\{v_n\in \Pi_1(x)\}}\le
\int_{\Pi_\eps} e^{-i s\cdot x} r(s)  e^{is\cdot v_n}  \, ds
\)
for $n\ge1$, $x\in\R^d$. 
\end{lemma}

\begin{proof}
Define $\gamma:\R^d\to[0,\infty)$ as in Section~\ref{subsec-stprlld}.
Since $\gamma\ge0$ and $\gamma|_{\Pi_2(0)}\ge1$,
\begin{align*}
1_{\{v_n\in \Pi_1(x)\}}
 & =\frac{1}{2^d} \int_{\Pi_1(x)} 1_{\{v_n\in \Pi_1(x)\}}\,dy
 \\ & \le \frac{1}{2^d}\int_{\Pi_1(x)} 1_{\{v_n\in \Pi_2(y)\}}\,dy
\le \frac{1}{2^d} \int_{\Pi_1(x)} \gamma(y-v_n)\,dy.
\end{align*}
Using the Fourier inversion formula~\eqref{eq-inv},
\begin{align*}
1_{\{v_n\in \Pi_1(x)\}} & 
\le \frac{1}{(4\pi)^d}\int_{\Pi_\eps} \Big(\int_{\Pi_1(x)} e^{-is\cdot y}\,dy\Big) \hat\gamma(s)e^{is\cdot v_n}\,ds
= \int_{\Pi_\eps} e^{-is\cdot x}r(s)e^{is\cdot v_n}\,ds
\end{align*}
by Fubini.
\end{proof}

\begin{pfof}{Theorem~\ref{thm-lldGM}}
By Lemma~\ref{lemma-gamma} and positivity of $R$,
\[
 R^n 1_{\{v_n\in \Pi_1(x)\}}
 \le 
\int_{\Pi_\eps} e^{-is\cdot x}r(s)R^n e^{is\cdot v_n}\,ds=
 A_{n,x}1
\]
where
\[
A_{n,x}=\int_{\Pi_\eps} e^{-is\cdot x}r(s)R(s)^n\,ds.
\]
Since $1\in\cB$,
\[
|R^n 1_{\{v_n\in \Pi_1(x)\}}|_\infty
\le |A_{n,x}1|_\infty
\le \|A_{n,x}1\| \le \|A_{n,x}\|\,\|1\|\ll \|A_{n,x}\|.
\]
Hence it suffices to estimate $\|A_{n,x}\|$.

Since $r$ is bounded and supported in $\Pi_\eps$, it follows that
$\|A_{n,x}\|\ll \int_{\Pi_\eps}\|R(s)\|^n\,ds$.
By~\eqref{eq-sp3} and Corollary~\ref{cor-eigvint}, 
\[
\|A_{n,x}\|\ll \int_{\Pi_\eps}|\lambda(s)|^n\,ds\ll a_n^{-d}.
\]
Hence, for $n\gg (1+|x|^\alpha)/\tilde{\ell}(|x|)$, we obtain the required estimate
$\|A_{n,x}\|\ll \frac{n}{a_n^d}\, \frac{\tilde\ell(|x|)}{1+|x|^\alpha}$.

As in the proof of Theorem~\ref{thm-lldgen}, it remains to prove that 
$\|A_{n,x}\|\ll \frac{n}{a_n^d}\frac{\tilde\ell(|x|)}{|x|^\alpha}$
for $a_n\le|x|$, 
$|x|\ge \pi/\eps$.

\vspace{2ex}
\noindent {\bf{The case $\alpha\in (0,1)$.}}  
Let $h=\pi x/|x|^2$.
The same modulus of continuity argument as in the i.i.d.\ case
(cf.\ \eqref{eq-mod}) yields
\(
\|A_{n,x}\|\le I_1+I_2
\)
where
\begin{align*}
I_1 & = \int_{\R^d} |r(s)-r(s-h)|\,\|R(s)^n\|\,ds
\ll |x|^{-1}\int_{\Pi_{2\eps}} \|R(s)^n\|\,ds, \\
I_2 & =\int_{\R^d} |r(s-h)|\,\|R(s)^n-R(s-h)^n\|\,ds
\ll \int_{\Pi_{2\eps}} \|R(s)^n-R(s-h)^n\|\,ds.
\end{align*}
By Corollary~\ref{cor-Rn}(i),
$I_1\ll |x|^{-1}a_n^{-d}\ll \frac{n}{a_n^d}\frac{\ell(|x|)}{|x|^\alpha}$.

Next, 
\[
\|R(s)^n-R(s-h)^n\|\le \sum_{k=0}^{n-1} \|R(s)^k\|\,\|R(s)-R(s-h)\|\,\|R(s-h)^{n-k-1}\|,
\]
so by (H2)(ii) and Corollary~\ref{cor-Rn}(i),
\[
I_2 \ll \frac{\ell(|x|)}{|x|^\alpha}\sum_{k=0}^{n-1}\int_{\Pi_{2\eps}} \|R(s)^k\|\,\|R(s-h)^{n-k-1}\|\,ds 
\ll \frac{n}{a_n^d}\frac{\ell(|x|)}{|x|^\alpha}.
\]

\vspace{1ex}
\noindent {\bf{The case $\alpha\in (1,2]$.}}  
Choose $j$ so that $|x_j|=\max\{|x_1|,\dots,|x_d|\}$.
Integrating by parts, $A_{n,x}=E_1+E_2$ where
\[
E_1=\frac{1}{ix_j} \int_{\R^d} e^{-is\cdot x}\partial_jr(s)R(s)^n\,ds,
\qquad
E_2=\frac{1}{ix_j} \int_{\R^d} e^{-is\cdot x}r(s)\partial_j(R(s)^n)\,ds.
\]
Integrating by parts once more and using that $r$ is $C^2$ and supported
in $\Pi_\eps$,
\begin{align*}
\|E_1\| & \le \frac{1}{x_j^2}\int_{\R^d} |\partial_j^2r(s)|\,\|R(s)^n\|\,ds
+
\frac{1}{x_j^2}\int_{\R^d} |\partial_jr(s)|\,\|\partial_j(R(s)^n)\|\,ds
\\ & \ll \frac{1}{x_j^2}\int_{\Pi_\eps} \|R(s)^n\|\,ds
+
 \frac{1}{x_j^2}\int_{\Pi_\eps} \|\partial_j(R(s)^n)\|\,ds.
\end{align*}
By Corollary~\ref{cor-Rn},
\[
\|E_1\| 
\ll \frac{1}{a_n^d}\frac{1}{|x|^2}+
 \frac{n}{a_n^d}\frac{1}{|x|^2}
\ll \frac{n}{a_n^d}\frac{\tilde\ell(|x|)}{|x|^\alpha}.
\]

Next, we exploit the modulus of continuity of $r\partial_j(R^n)$, writing
$h=\pi x_j^{-1}e_j$ and 
\begin{align*}
\|E_2\|
& \ll \frac{1}{|x_j|}\int_{\R^d} |r(s)-r(s-h)|\,\|\partial_j(R(s)^n)\|\,ds
\\ & \qquad\qquad\qquad +\frac{1}{|x_j|}\int_{\R^d} |r(s-h)|\,\|\partial_j(R(s)^n)-\partial_j(R(s-h)^n)\|\,ds
\\ & \ll \frac{1}{|x|^2}\int_{\Pi_{2\eps}} \|\partial_j(R(s)^n)\|\,ds
+\frac{1}{|x|}\int_{\Pi_{2\eps}} \|\partial_j(R(s)^n)-\partial_j(R(s-h)^n)\|\,ds.
\end{align*}
Again $\frac{1}{|x|^2}\int_{\Pi_{2\eps}} \|\partial_j(R(s)^n)\|\,ds\ll  \frac{n}{a_n^d}\frac{1}{|x|^2}
\ll \frac{n}{a_n^d}\frac{\tilde\ell(|x|)}{|x|^\alpha}$ so
it remains to estimate
\[
J=\frac{1}{|x|}\int_{\Pi_{2\eps}} \|\partial_j(R(s)^n)-\partial_j(R(s-h)^n)\|\,ds.
\]
By~\eqref{eq-sp},
\[\partial_j
(R(s)^n)=n\lambda(s)^{n-1}\partial_j\lambda(s)P(s)+\lambda(s)^nP'(s)+\partial_j(Q(s)^n).
\]
Relabel $\{s,s-h\}=\{s_1,s_2\}$ where $|\lambda(s_1)|\le|\lambda(s_2)|$.  Then
$J\le F_1+\dots+F_6$ where
\begin{align*}
F_1 & = \frac{n}{|x|} \int_{\Pi_{2\eps}} 
|\lambda(s_1)^{n-1}-\lambda(s_2)^{n-1}|\,
|\partial_j\lambda(s_2)|\,\|P(s_2)\|\,ds,
\\ F_2  & = \frac{n}{|x|} \int_{\Pi_{2\eps}} |\lambda(s_2)|^{n-1}
|\partial_j\lambda(s_1)-\partial_j\lambda(s_2)|
\|P(s_2)\|\,ds,
\\ F_3  & = \frac{n}{|x|} \int_{\Pi_{2\eps}} 
|\lambda(s_2)|^{n-1}|\partial_j\lambda(s_2)|\|P(s_1)-P(s_2)\|
\,ds,
\\ F_4 & = \frac{1}{|x|}\int_{\Pi_{2\eps}} |\lambda(s_1)^n-\lambda(s_2)^n|\|P'(s_2)\|\,ds,
\\ F_5 & = \frac{1}{|x|}\int_{\Pi_{2\eps}} |\lambda(s_2)|^n\|P'(s_1)-P'(s_2)\|\,ds,
\\ F_6  & = \frac{1}{|x|}\int_{\Pi_{2\eps}} \|\partial_j(Q^n)(s_1)-\partial_j(Q^n)(s_2)\|\,ds.
\end{align*}

The hardest term $F_1$ is estimated in the same way as $J_2$ in the proof of Theorem~\ref{thm-lldgen} so we write the calculation without the justifications:
\begin{align*}
F_1  & \ll \frac{n^2}{|x|} \int_{\Pi_{2\eps}} |\lambda(s_1)-\lambda(s_2)|\,|\lambda(s_2)|^{n-2}|\partial_j\lambda(s_2)|\,ds
\\ & \ll \frac{n^2}{|x|^2} \int_{\Pi_{2\eps}} |\partial_j\lambda(s^*)|\,
|\partial_j\lambda(s_2)|\, |\lambda(s_2)|^{n-2} \,ds
\\ & \ll  \frac{n^2}{|x|^2}\int_{\Pi_{2\eps}} |\partial_j\lambda(s^*)-\partial_j\lambda(s_2)|\,
|\partial_j\lambda(s_2)| \,|\lambda(s_2)|^{n-2} \,ds
+  \frac{n^2}{|x|^2} \int_{\Pi_{2\eps}} 
|\partial_j\lambda(s_2)|^2 |\lambda(s_2)|^{n-2} \,ds
\\ & \ll \frac{n^2\tilde\ell(|x|)}{|x|^{\alpha+1}} \int_{\Pi_{3\eps}}
|s|^{\alpha-1}\tilde\ell(1/|s|)|\lambda(s)|^{n-2} \,ds
+ \frac{n^2}{|x|^2} \int_{\Pi_{3\eps}}
|s|^{2(\alpha-1)}\tilde\ell(1/|s|)^2 |\lambda(s)|^{n-2} \,ds
\\ &  \ll
\frac{n}{a_n^d}\frac{\tilde\ell(|x|)}{|x|^\alpha}\Big(\frac{a_n}{|x|}+\frac{a_n^{2-\alpha}\tilde\ell(a_n)}{|x|^{2-\alpha}\tilde\ell(|x|)}\Big)
\ll \frac{n}{a_n^d}\frac{\tilde\ell(|x|)}{|x|^\alpha}.
\end{align*}
This is the only term that requires Lemma~\ref{lemma-eigv}(ii).
The terms $F_2,\dots,F_5$ require only the rougher estimates in Lemma~\ref{lemma-PQ} combined with Corollary~\ref{cor-eigvint} and 
we obtain
\[
F_2 \ll \frac{n}{a_n^d}\frac{\tilde\ell(|x|)}{|x|^\alpha},
\quad
F_3 ,\, F_4 \ll \frac{n}{a_n^d}\frac{1}{|x|^2}, \quad
F_5 \ll  \frac{1}{a_n^d}\frac{\tilde\ell(|x|)}{|x|^\alpha}.
\]
Finally, by Lemma~\ref{lemma-op},
$F_6 \ll \delta_1^n\frac{\tilde\ell(|x|)}{|x|^\alpha}$
which ends the proof.
\end{pfof}

\subsection{Gibbs-Markov maps}
\label{sec-GM}

Let $(\Lambda,\mu)$ be a probability space with an at most countable measurable partition $\{\Lambda_k\}$, and let $f:\Lambda\to \Lambda$ be an ergodic measure-preserving transformation.
Define $s(z,z')$ to be the least integer $n\ge0$ such that $f^nz$ and $f^nz'$ lie in distinct partition elements.
It is assumed that $s(z,z')=\infty$ if and only if $z=z'$;  then $d_\theta(z,z')=\theta^{s(z,z')}$ is a metric
for $\theta\in(0,1)$, 

Let $g=\frac{d\mu}{d\mu\circ f}:\Lambda\to\R$.
We say that $f$ is a {\em Gibbs-Markov map} if
\begin{itemize}

\parskip = -2pt
\item $f \Lambda_k$ is a union of partition elements and $f|_{\Lambda_k}:\Lambda_k\to f\Lambda_k$ is a measurable bijection for each $k\ge1$;
\item $\inf_k\mu(f\Lambda_k)>0$;
\item
There are constants $C>0$, $\theta\in(0,1)$ such that
$|\log g(z)-\log g(z')|\le Cd_\theta(z,z')$ for all $z,z'\in \Lambda_k$, $k\ge1$.
\end{itemize}
Standard references for Gibbs-Markov maps include~\cite{AD01,ADU93}.  

Given $\phi:\Lambda\to\R$, let
\[
D_k\phi=\sup_{z,z'\in\Lambda_k,\,z\neq z'}|\phi(z)-\phi(z')|/d_\theta(z,z'),\qquad |\phi|_\theta=\sup_{k\ge1}D_k\phi.
\]
We define the Banach space $\cF_\theta\subset L^\infty$ to consist of functions
$\phi:\Lambda\to\R$ such that $|\phi|_\theta<\infty$ with norm
$\|\phi\|_\theta=|\phi|_\infty+|\phi|_\theta<\infty$.
For $\phi:\Lambda\to\R^d$, define $|\phi|_\theta=\max_{j=1,\dots,d}|\phi_j|_\theta$.

\begin{prop} \label{prop-GM}
Assume $f$ is a mixing Gibbs-Markov map and
let $v:\Lambda\to\R^d$ with $\int_\Lambda |v|^2\,d\mu=\infty$ and $|v|_\theta<\infty$.
Fix $\alpha\in(0,1)\cup(1,2]$ and 
assume that $v$ satisfies (H1).  

Then conditions (H1)--(H3) are satisfied with Banach space $\cB=\cF_\theta$.
\end{prop}

\begin{proof}
Condition (H1) is satisfied by assumption and condition~(H3) is well-known for mixing Gibbs-Markov maps~\cite{AD01,ADU93}.
It remains to verify that (H2) holds.
In fact, for any $M>0$ the conditions in (H2) hold for all
 $|s|\le M$, $|h|\le1$.
We verify this for (H2)(iii).
All the other calculations are simpler and hence omitted.

Now $(\partial_j R(s+h)-\partial_j R(s))\phi=iR(\phi\psi)$ where
$\psi=v_je^{is\cdot v}(e^{ih\cdot v}-1)$.
A standard calculation shows that 
\[
\|R(\phi\psi)\|_\theta
\ll \sumk\mu(\Lambda_k)
\big(\supk|\phi\psi|+D_k(\phi\psi)\big)
\le \|\phi\|_\theta \sumk\mu(\Lambda_k)
\big(2\,\supk|\psi|+D_k\psi\big).
\]
where $\supk=\sup_{\Lambda_k}$ and $\inf_k=\inf_{\Lambda_k}$.
Hence
\[
\|\partial_jR(s+h)-\partial_jR(s)\|_\theta\ll \sumk \mu(\Lambda_k)\big\{\supk|v_j(e^{ih\cdot v}-1)|+
D_k\big(ve^{is\cdot v}(e^{ih\cdot v}-1)\big)\big\}.
\]
Also,  $D_k e^{is\cdot v}\le |s||v|_\theta\ll|s|$, so
\(
\|\partial_jR(s+h)-\partial_jR(s)\|_\theta\ll S_1+S_2+S_3+S_4,
\)
where
\begin{align*}
S_1 & = \sumk \mu(\Lambda_k)\supk|v_j(e^{ih\cdot v}-1)|, \\[.75ex]
S_2 & = \sumk \mu(\Lambda_k)\supk|v_j(e^{ih\cdot v}-1)| D_k e^{is\cdot v} 
\ll |s|S_1\le MS_1, \\[.75ex]
S_3 & = \sumk \mu(\Lambda_k)\supk|v| D_k e^{ih\cdot v}
 \le |h|\sumk \mu(\Lambda_k)\supk|v|, \\[.75ex]
S_4 & = \sumk \mu(\Lambda_k)\supk|e^{ih\cdot v}-1| D_k v \ll    S_3.
\end{align*}

Next,
$\supk|v|- \infk|v|
\le |v|_\theta\ll1$.  Hence 
\[
\sumk \mu(\Lambda_k)\supk|v|
\ll \sumk \mu(\Lambda_k)(1+\infk|v|)
\le 1+\int_\Lambda|v|\,d\mu,
\]
and we obtain $S_3,\,S_4\ll |h|$.

Finally,
\[
\supk|v_j(e^{ih\cdot v}-1)|\ll
\big(1+\infk|v|\big)
\big(|h|+\infk|e^{ih\cdot v}-1|\big)\ll 
|h|\big(1+\infk|v|\big)+\infk|v_j(e^{ih\cdot v}-1)|,
\]
and so
\begin{align*}
S_1 & \ll  \sumk \mu(\Lambda_k) \big(|h|(1+\infk|v|)+\infk|v_j(e^{ih\cdot v}-1)|\big)
\\ &   \le  |h|\Big(1+\int_\Lambda|v|\,d\mu\Big)+
\int_\Lambda|v_j(e^{ih\cdot v}-1)|\,d\mu
  \ll  |h|+ \int_\Lambda|v_j(e^{ih\cdot v}-1)|\,d\mu.
\end{align*}
The conditions on $v$ are the same as those on $X$ in Theorem~\ref{thm-lldgen}, so
\[
\int_\Lambda|v_j(e^{ih\cdot v}-1)|\,d\mu=\E|X_j(e^{ih\cdot X}-1)|
\ll |h|^{\alpha-1}\tilde\ell(1/|h|)
\]
by the proof of Lemma~\ref{lemma-chfacts}(ii).
Hence 
$S_1,\,S_2\ll  |h|+ |h|^{\alpha-1}\tilde\ell(1/|h|)$.

Altogether,
$\|\partial_j R(s+h)-\partial_j R(s)\|_\theta\ll 
|h|+|h|^{\alpha-1}\tilde\ell(1/|h|) \ll
|h|^{\alpha-1}\tilde\ell(1/|h|)$ as required.~
\end{proof}

\subsection{AFU maps}
\label{sec-AFU}

Let $\Lambda=[0,1]$ with measurable partition $\{I\}$ consisting of open intervals.  A map
$f:\Lambda\to \Lambda$ is called AFU if $f|_I$ is $C^2$ and strictly monotone for each $I$, and
\begin{itemize}

\parskip = -2pt
\item[(A)] (Adler's condition)   $f''/(f')^2$ is bounded on $\bigcup I$.
\item[(F)] (finite images)  The set of images $\{fI\}$ is finite.
\item[(U)] (uniform expansion)  There exists $\rho>1$ such that
$|f'|\ge\rho$ on $\bigcup I$.
\end{itemize}
A standard reference for such maps 
is~\cite{Zweimuller98} (see also~\cite{ADSZ04}).  Since AFU maps are not necessarily Markov, the H\"older spaces $\cF$ are not preserved by the transfer operator of $f$ and it is standard to consider the space of bounded variation functions.
Accordingly, we define the Banach space $\cB = \BV\subset L^\infty$
to consist of functions $\phi:\Lambda\to\R$ such that
$\Var\phi<\infty$
with norm $\|\phi\|=|\phi|_\infty+\Var \phi$.
Here 
\[
\Var \phi=\sup_{0=z_0<\dots<z_k=1}\sum_{i=1}^k|\phi(z_i)-\phi(z_{i-1})|
\]
 denotes the variation of $\phi$ on $\Lambda$.
Also, we let $\Var_I\phi$ denote the variation of $\phi$ on $I$.
For $\phi:\Lambda\to\R^d$, define $\Var \phi=\max_{j=1,\dots,d}\Var\phi_j$.

We suppose that $f:\Lambda\to \Lambda$ is topologically mixing.
Then there is a unique absolutely continuous $f$-invariant probability measure $\mu$, and $\mu$ is mixing.

\begin{prop} \label{prop-AFU}
Assume $f$ is a topologically mixing AFU map and
let $v:\Lambda\to\R^d$ with $\int_\Lambda |v|^2\,d\mu=\infty$ and $\supI\Var_Iv<\infty$.
Fix $\alpha\in(0,1)\cup(1,2]$ and 
assume that the tails of $v$ satisfy (H1).  

Then conditions (H1)--(H3) are satisfied with Banach space $\cB=\BV$.
\end{prop}

\begin{proof}
The proof essentially goes word for word as the proof of Proposition~\ref{prop-GM} with minor changes.
Condition (H1) is satisfied by assumption and condition~(H3) is well-known for mixing AFU maps.
It remains to verify that (H2) holds.
Fix $M>0$.  
As before, we verify (H2)(iii)
 for all $|s|\le M$, $|h|\le 1$;
the other calculations being simpler.

Again, $(\partial_j R(s+h)-\partial_j R(s))\phi=iR(\phi\psi)$ where
$\psi=v_je^{is\cdot v}(e^{ih\cdot v}-1)$, and
a standard calculation shows that
\begin{align*}
\|R(\phi\psi)\| \ll \|\phi\| \sumI \mu(I)(\supI |\psi|+\Var_I\psi).
\end{align*}
Also,
$\Var_I e^{is\cdot v}\le |s|\Var_I v\ll |s|$,
so $\|R'(s+h)-R'(s)\|\le S_1+S_2+S_3+S_4$ where
\begin{align*}
S_1 & =\sumI \mu(I)\supI|v_j(e^{ih\cdot v}-1)| 
\\
S_2 & =\sumI \mu(I)\supI|v_j(e^{ih\cdot v}-1)|\Var_I e^{is\cdot v} 
\ll |s|S_1\le MS_1,
\\
S_3 & = \sumI \mu(I)\supI|v| \Var_I e^{ih\cdot v}
 \le |h|\sumI \mu(I)\supI|v|, \\[.75ex]
S_4 & = \sumI \mu(I)\supI|e^{ih\cdot v}-1| \Var_I v \ll    S_3.
\end{align*}
The calculation continues exactly as in Proposition~\ref{prop-GM} and we omit the remaining details.
\end{proof}

\paragraph{Acknowledgements}
DT was partially supported by EPSRC grant EP/S019286/1.
We are grateful to the referees for several helpful comments and suggestions.

\end{document}